\documentclass{article}

 \PassOptionsToPackage{numbers, compress}{natbib}

 \usepackage[preprint]{nips_2018}

\usepackage[utf8]{inputenc} 
\usepackage[T1]{fontenc}    
\usepackage{hyperref}       
\usepackage{url}            
\usepackage{booktabs}       
\usepackage{amsfonts}       
\usepackage{nicefrac}       
\usepackage{microtype}      

\usepackage[utf8]{inputenc} 
\usepackage[english]{babel} 

\usepackage{graphicx}
\usepackage{mathtools, cuted}
\usepackage{hyperref}
\usepackage{amsmath}
\usepackage{amssymb}
\usepackage{mathtools}
\usepackage{amsthm}

\newtheorem{theorem}{Theorem}[section]
\newtheorem{definition}[theorem]{Definition}

\newtheorem{proposition}[theorem]{Proposition} 
\newtheorem{remark}[theorem]{Remark}

\title{Sample-Path Equivalent CM Models}

\author{
  Reza Rezaie and X. Rong Li \\
 Department of Electrical Engineering\\
 University of New Orleans\\
New Orleans, LA 70148 \\
  \texttt{rrezaie@uno.edu} and \texttt{xli@uno.edu} \\
}

\begin{document}

\maketitle

\begin{abstract}
The conditionally Markov (CM) sequence contains different classes including Markov, reciprocal, and so-called $CM_L$ and $CM_F$ (two special classes of CM sequences). Each class has its own forward and backward dynamic models. The evolution of a CM sequence can be described by different models. For example, a Markov sequence can be described by a Markov model, as well as by reciprocal, $CM_L$, and $CM_F$ models. Also, sometimes a forward model is available, but it is desirable to have a backward model for the same sequence (e.g., in smoothing). Therefore, it is important to study relationships between different dynamic models of a CM sequence. This paper discusses such relationships between models of nonsingular Gaussian (NG) $CM_L$, $CM_F$, reciprocal, and Markov sequences. Two models are said to be explicitly sample-equivalent if not only they govern the same sequence, but also a one-one correspondence between their sample paths is made explicitly. A unified approach is presented, such that given a forward/backward $CM_L$/$CM_F$/reciprocal/Markov model, any explicitly equivalent model can be obtained. As a special case, a backward Markov model explicitly equivalent to a given forward Markov model can be obtained regardless of the singularity/nonsingularity of the state transition matrix of the model.

\end{abstract}


\section{Introduction}

Markov processes have been widely used in many different applications for modeling random phenomena. In some problems more general stochastic processses (e.g., reciprocal processes or CM processes) are required. CM and reciprocal processes have been used in many different applications, including stochastic mechanics, intent inference, image processing, trajectory modeling, and acausal systems \cite{Levy_1}--\cite{Krener1}. Dynamic models of CM, reciprocal, and Markov processes play a very important role in the application of these processes. This paper elaborates on the relationship between these models. 

Gaussian CM processes were introduced in \cite{Mehr}. Reciprocal processes were introduced in \cite{Bernstein} related to a problem posed by E. Schrodinger \cite{Schrodinger_1}--\cite{Schrodinger_2}. Later, reciprocal processes were studied more in \cite{Slepian}--\cite{Moura}. Dynamic models and characterizations of NG $CM_L$ and $CM_F$ sequences were obtained in \cite{CM_Part_I_Conf}. A Gaussian sequence is reciprocal if and only if (iff) it is both $CM_L$ and $CM_F$ \cite{CM_Part_II_A_Conf}. A dynamic model with locally correlated dynamic noise governing the NG reciprocal sequence was presented in \cite{Levy_Dynamic}. Dynamic models with white dynamic noise governing the NG reciprocal sequence were developed in \cite{CM_Part_II_A_Conf}.

Consider stochastic sequences defined over $[0,N]=\lbrace 0,1,\ldots,N \rbrace$. For convenience, let the index be time. A sequence is Markov iff conditioned on the state at any time $k$, the subsequences before and after $k$ are independent. A sequence is reciprocal iff conditioned on the states at any two times $k_1$ and $k_2$, the subsequences inside and outside the interval $[k_1,k_2]$ are independent. In other words, ``inside" and ``outside" are independent given the boundaries. A sequence is $CM_L$ ($CM_F$) iff conditioned on the state at time $N$ ($0$), the sequence is Markov over $[0,N-1]$ ($[1,N]$). The subscripts ``$L$" (``$F$") is used because the conditioning is at the \textit{last} (\textit{first}) time of the interval.  

The Markov sequence is a special class of reciprocal sequences. The reciprocal sequence is a special class of $CM_L$/$CM_F$ sequences. Thus, evolution of a Markov sequence can be governed by a Markov/reciprocal/$CM_L$/$CM_F$ model. Similarly, evolution of a reciprocal sequence can be governed by a reciprocal/$CM_L$/$CM_F$ model. Therefore, a CM sequence can have more than one model. These models are \textit{equivalent} in the sense that they govern the same sequence (i.e., the sequences governed by the models have the same distribution). Also, sometimes only a forward (backward) model is available when the corresponding backward (forward) one is desired or required. These forward and backward models are also equivalent since they govern the same sequence. In some cases, the above definition of equivalent models is not sufficient because it is only about distributions, not each sample path. It is only equivalent for the set, not element-wise equivalent. The two-filter smoothing approach is an example, where to verify the conditions required for derivation, one needs relationship between sample paths of dynamic noises and boundary values of forward and backward Markov models for the same sample path of the sequence \cite{Wax_Kailath}--\cite{Alan_Wilskey}. Given a model and a sample path of its dynamic noise and boundary values\footnote{For a forward (backward) Markov model, boundary values mean the initial (final) values.} corresponding to an arbitrary sample path of the sequence, it is desirable to obtain an equivalent model and a sample path of its dynamic noise and boundary values leading to the same sample path of the sequence. In other words, it is desirable to find relationship between the dynamic noises and boundary values (of two equivalent models) leading to the same sample path of the governed sequence. Therefore, such models with an explicit relationship between their sample paths are said to be \textit{explicitly sample-equivalent}. It is important to find relationships between the equivalent models because one model can be more easily applicable than the other in some applications. For example, the reciprocal model of \cite{Levy_Dynamic} is driven by colored noise and not necessarily easy to apply for trajectory modeling \cite{DD_Conf}--\cite{DW_Conf}. But the equivalent recirpocal $CM_L$ model of \cite{CM_Part_II_A_Conf} is driven by white noise and its application is straightforward. 

Determination of a NG backward Markov model based on its forward model has been the topic of several papers \cite{MB_1}--\cite{Verghese}. Equivalence of the backward Markov model in \cite{MB_1}--\cite{MB_4} was derived based on equality of the second moments calculated by forward and backward models, which does not deal with specific sample paths. \cite{Verghese} presented an explicitly equivalent backward Markov model only for forward models with nonsingular state transition matrices. In the case of a singular state transition matrix, the derivation of \cite{Verghese} does not provide explicit equivalence (i.e., the relationship between sample paths of the dynamic noises and the boundary values) of forward and backward models. In such a case, \cite{Verghese} only provides parameters of the backward model. Given a NG Markov model, an approach was presented in \cite{Levy_Dynamic} for the determination of an explicitly equivalent reciprocal model with locally correlated dynamic noise governing the same Markov sequence. 

The main contributions of this paper are as follows. It discusses relationships between dynamic models governing NG $CM_L$, $CM_F$, reciprocal, and Markov sequences. A unified approach is presented, such that within these classes given a model, any explicitly equivalent model can be obtained. As a special case, a backward Markov model explicitly equivalent to a forward Markov model can be obtained. Unlike \cite{Verghese}, this approach works for both singular and nonsingular state transition matrices. The explicitly equivalent reciprocal model obtained in \cite{Levy_Dynamic} can be derived by our approach. 

Section \ref{Definitions} reviews definitions and models of $CM_L$, $CM_F$, reciprocal, and Markov sequences. Also, definition of explicitly sample-equivalent models is presented. Section \ref{General_Approach} presents an approach to determining explicitly equivalent models. In Section \ref{Examples}, explicitly sample-equivalent forward and backward Markov models, and explicitly sample-equivalent $CM_L$ and reciprocal models are obtained. Section \ref{Summary} contains a summary and conclusions.

\section{Definitions and Preliminaries}\label{Definitions}

Throughout the paper we consider sequences defined over $[0,N]$. The following conventions are used:
\begin{align*}
[i,j]& \triangleq \lbrace i,i+1,\ldots ,j-1,j \rbrace, i,j \in [0,N], i<j\\
[x_k]_{i}^{j} & \triangleq \lbrace x_k, k \in [i,j] \rbrace\\
 [x_k] & \triangleq [x_k]_{0}^{N}\\
x & \triangleq [x_0, \ldots , x_N]'
\end{align*}
Also, $C$ is the covariance matrix of the whole sequence $[x_k]$. The symbol ``$\setminus$" is used for set subtraction.

\begin{definition}
Two dynamic models are \textit{equivalent} if they govern the same sequence (i.e., their sequences have the same distribution).

\end{definition}

\begin{definition}\label{Explicit_Equivalent}
Two dynamic models are \textit{explicitly sample-equivalent} if, given a sample path of the dynamic noise and the boundary values of one model for an arbitrary sample path of the governed sequence, a sample path of the dynamic noise and the boundary values of the other model leading to the same sample path of the governed sequence is given explicitly. 

\end{definition}

In other words, if the relation between sample paths of dynamic noises and boundary values of two equivalent models is provided (so that the two models generate the same sample path), they become explicitly equivalent.

Forward and backward Markov, reciprocal, $CM_L$, and $CM_F$ models of \cite{Levy_Dynamic}, \cite{CM_Part_I_Conf}, \cite{CM_Part_II_A_Conf} are reviewed first. Let $[x_k]$ be a zero-mean NG sequence.

\subsection{Markov Model}
$[x_k]$ is Markov iff
\begin{align}
x_k&=M_{k,k-1}x_{k-1}+e^M_{k}, k \in [1,N] \label{Markov_Dynamic_Forward}\\
x_0&=e^M_0\label{M_BC}
\end{align}
where $[e^M_k]$ is a zero-mean white NG sequence with covariances $M_k$.  

We have
\begin{align*}
\mathcal{M}x&=e^M\\
e^M &= [(e^M_0)' , \ldots , (e^M_N)']'
\end{align*} 
where $\mathcal{M}$ is the nonsingular matrix
\begin{align}\label{M}
\left[ \begin{array}{cccccc}
I & 0 & 0 &  \cdots & 0 & 0\\
-M_{1,0} & I & 0 &  \cdots & 0 & 0\\
0 & -M_{2,1} & I & 0 & \cdots & 0\\
\vdots & \vdots & \vdots & \vdots & \vdots & \vdots \\
0 & 0 & \cdots & -M_{N-1,N-2} & I & 0\\
0 & 0 & 0 &  \cdots & -M_{N,N-1} & I
\end{array}\right]
\end{align}

A NG sequence with covariance matrix $C$ is Markov iff $C^{-1}$ is tri-diagonal given by $\eqref{CML}$ below with $D_0=\cdots=D_{N-2}=0$.

\subsection{Backward Markov Model}
$[x_k]$ is Markov iff
\begin{align}
x_{k}&=M^B_{k,k+1}x_{k+1}+e^{MB}_{k}, k \in [0,N-1]  \label{Markov_Dynamic_Backward}\\
x_N&=e^{MB}_N\label{MB_BC}
\end{align}
where $[e^{MB}_k]$ is a zero-mean white NG sequence with covariances $M^B_k$. 

We have
\begin{align*}
\mathcal{M}^Bx&=e^{MB}\\
e^{MB} &=[(e^{MB}_0)' , \ldots , (e^{MB}_{N})']'
\end{align*} 
where $\mathcal{M}^B$ is the nonsingular matrix 
\begin{align}\label{MB}
\left[ \begin{array}{cccccc}
I & -M^{B}_{0,1} & 0 &  \cdots & 0 & 0\\
0 & I & -M^B_{1,2} & 0 &  \cdots & 0\\
0 & 0 & I & -M^B_{2,3} & \cdots & 0\\
\vdots & \vdots & \vdots & \vdots & \vdots & \vdots \\
0 & 0 & \cdots & 0 & I & -M^B_{N-1,N}\\
0 & 0 & 0 &  \cdots & 0 & I
\end{array}\right]
\end{align}

\subsection{Reciprocal Model}
$[x_k]$ is reciprocal iff 
\begin{align}
R^0_kx_k-R^-_{k}x_{k-1}-R^+_{k}x_{k+1}=e^R_k, k \in [1,N-1] \label{Reciprocal_Dynamic}
\end{align}
where $[e^R_k]_0^{N}$ is a zero-mean NG sequence with $E[e^R_k(e^R_{k})']=R^0_k$, $E[e^R_k(e^R_{k+1})']=-R^+_k$, $E[e^R_k(e^R_j)']=0, |k-j|>1$, $R^+_k=(R^-_{k+1})'$, and the boundary conditions
\begin{align}
R^0_0x_0-R^-_0x_N-R^+_0x_1&=e^R_0\label{Reciprocal_Cyclic_BC1}\\
R^0_Nx_N-R^-_Nx_{N-1}-R^+_Nx_0&=e^R_N\label{Reciprocal_Cyclic_BC2}
\end{align}
where $E[e^R_N(e^R_0)']=-R^+_N$ and $R^-_0=(R^+_N)'$, and the parameters of model $\eqref{Reciprocal_Dynamic}$ and boundary conditions $\eqref{Reciprocal_Cyclic_BC1}$--$\eqref{Reciprocal_Cyclic_BC2}$ lead to a nonsingular sequence. 

We have
\begin{align*}
\mathcal{R}x&=e^R\\
e^R &= [(e^R_0)' , \ldots , (e^R_N)']'
\end{align*} 
where $\mathcal{R}$ is the nonsingular matrix
\begin{align}\label{R}
\left[ \begin{array}{cccccc}
R^0_0 & -R^+_0 & 0 &  \cdots & 0 & -R^-_0\\
-R^-_{1} & R^0_1 & -R^+_1 & 0 & \cdots & 0\\
0 & -R^-_{2} & R^0_2 & -R^-_2 & \cdots & 0\\
\vdots & \vdots & \vdots & \vdots & \vdots & \vdots \\
0 & 0 & \cdots & -R^-_{N-1} & R^0_{N-1} & -R^+_{N-1}\\
-R^+_N & 0 & 0 &  \cdots & -R^-_{N} & R^0_N
\end{array}\right]
\end{align}

A NG sequence with covariance matrix $C$ is reciprocal iff its $C^{-1}$ is cyclic tri-diagonal given by $\eqref{CML}$ with $D_1=\cdots=D_{N-2}=0$.

The reciprocal model $\eqref{Reciprocal_Dynamic}$ (with its boundary conditions) is well-posed if its parameters lead to a nonsingular covariance matrix for the whole sequence $[x_k]$. 

We call model $\eqref{Reciprocal_Dynamic}$ the reciprocal model, to distinguish it from reciprocal $CM_L$/$CM_F$ models, defined below.

\subsection{Forward $CM_c$ Models}
$[x_k]$ is $CM_c$ iff
\begin{align}
x_k&=G_{k,k-1}x_{k-1}+G_{k,c}x_c+e_k, k \in [1,N] \setminus \lbrace c \rbrace
\label{CML_Dynamic_Forward}\\
x_c&=e_c, \quad x_0=G_{0,c}x_c+e_0 \, \, (\text{for} \, \, c=N) \label{CML_Forward_BC2}
\end{align}
where $[e_k]$ is a zero-mean white NG sequence with covariances $G_k$.

For $c=0$, 
\begin{align}
\mathcal{G}^Fx&=e^F\label{Lx=e}\\
e^F& = [e_0'  , \ldots ,  e_N']'\nonumber
\end{align}
where $\mathcal{G}^F$ is the nonsingular matrix 
\begin{align}\label{F}
\left[ \begin{array}{cccccc}
I & 0 & 0 &  \cdots & 0 & 0\\
-2G_{1,0} & I & 0 &  \cdots & 0 & 0\\
-G_{2,0} & -G_{2,1} & I & 0 & \cdots & 0\\
\vdots & \vdots & \vdots & \vdots & \vdots & \vdots \\
-G_{N-1,0} & 0 & \cdots & -G_{N-1,N-2} & I & 0\\
-G_{N,0} & 0 & 0 &  \cdots & -G_{N,N-1} & I
\end{array}\right]
\end{align}

For $c=N$,
\begin{align}
\mathcal{G}^Lx&=e^L\label{Lx=e}\\
e^L& =[e_0' , \ldots ,  e_N']'\nonumber
\end{align}
where $\mathcal{G}^L$ is the nonsingular matrix
\begin{align}\label{L_2}
\left[ \begin{array}{cccccc}
I & 0 & 0 &  \cdots & 0 & -G_{0,N}\\
-G_{1,0} & I & 0 &  \cdots & 0 & -G_{1,N}\\
0 & -G_{2,0} & I & 0 & \cdots & -G_{2,N}\\
\vdots & \vdots & \vdots & \vdots & \vdots & \vdots \\
0 & 0 & \cdots & -G_{N-1,N-2} & I & -G_{N-1,N}\\
0 & 0 & 0 &  \cdots & 0 & I
\end{array}\right]
\end{align}

A NG sequence with covariance matrix $C$ is $CM_L$ ($CM_F$) iff its $C^{-1}$ is $CM_L$ ($CM_F$), defined as follows.

\begin{definition}
A symmetric positive definite matrix is $CM_L$ if it has form $\eqref{CML}$ and $CM_F$ if it has form $\eqref{CMF}$:
\begin{align}
\left[
\begin{array}{ccccccc}
A_0 & B_0 & 0 & \cdots & 0 & 0 & D_0\\
B_0' & A_1 & B_1 & 0 & \cdots & 0 & D_1\\
0 & B_1' & A_2 & B_2 & \cdots & 0 & D_2\\
\vdots & \vdots & \vdots & \vdots & \vdots & \vdots & \vdots\\
0 & \cdots & 0 & B_{N-3}' & A_{N-2}  & B_{N-2} & D_{N-2}\\
0 & \cdots & 0 & 0 & B_{N-2}' & A_{N-1} & B_{N-1}\\
D_0' & D_1' & D_2' & \cdots & D_{N-2}' & B_{N-1}' & A_N
\end{array}\right]\label{CML}\\
\left[
\begin{array}{ccccccc}
A_0 & B_0 & D_2 & \cdots & D_{N-2} & D_{N-1} & D_{N}\\
B_0' & A_1 & B_1 & 0 & \cdots & 0 & 0\\
D_2' & B_1' & A_2 & B_2 & \cdots & 0 & 0\\
\vdots & \vdots & \vdots & \vdots & \vdots & \vdots & \vdots\\
D_{N-2}' & \cdots & 0 & B_{N-3}' & A_{N-2}  & B_{N-2} & 0\\
D_{N-1}' & \cdots & 0 & 0 & B_{N-2}' & A_{N-1} & B_{N-1}\\
D_{N}' & 0 & 0 & \cdots & 0 & B_{N-1}' & A_N
\end{array}\right]\label{CMF}
\end{align}

\end{definition}
Here $A_k$, $B_k$, and $D_k$ are matrices in general.  

\begin{remark}
$[x_k]$ is reciprocal iff it obeys $\eqref{CML_Dynamic_Forward}$--$\eqref{CML_Forward_BC2}$ and 
\begin{align}
G_k^{-1}G_{k,c}=G_{k+1,k}'G_{k+1}^{-1}G_{k+1,c}
\label{CML_Condition_Reciprocal}
\end{align}
$\forall k \in [1,N-2]$ for $c=N$, and $\forall k \in [2,N-1]$ for $c=0$. Moreover, $[x_k]$ is Markov iff we also have, for $c=N$, 
\begin{align}
G_0^{-1}G_{0,N}=G_{1,0}'G_1^{-1}G_{1,N}
\end{align}
and for $c=0$, 
\begin{align}
G_{N,0}=0
\end{align}

\end{remark}

\subsection{Backward $CM_c$ Models}
$[x_k]$ is $CM_c$ iff
\begin{align}
x_k&=G^B_{k,k+1}x_{k+1}+G^B_{k,c}x_c+e^{B}_k, k \in [0,N-1] \setminus \lbrace c\rbrace
\label{CML_Dynamic_Backward}\\
x_c&=e^{B}_c, \quad x_{N}=G^{B}_{N,c}x_c+e^{B}_{c} \, \, (\text{for} \, \, c=0) \label{CML_Backward_BC2}
\end{align}
where $[e^{B}_k]$ is a zero-mean white NG sequence with covariances $G^B_k$.

\begin{remark}
$[x_k]$ is reciprocal iff it obeys $\eqref{CML_Dynamic_Backward}$--$\eqref{CML_Backward_BC2}$ and
\begin{align}
(G^B_{k+1})^{-1}G^B_{k+1,c}=(G^B_{k,k+1})'(G^B_{k})^{-1}G^B_{k,c}
\label{CMF_Condition_Reciprocal_B}
\end{align}
$\forall k \in [1,N-2]$ for $c=0$, and $\forall k \in [0,N-3]$ for $c=N$. Moreover, $[x_k]$ is Markov iff we also have, for $c=0$, 
\begin{align}
G^B_N)^{-1}G^{B}_{N,0}=(G^B_{N-1,N})'(G^B_{N-1})^{-1}G^B_{N-1,0} 
\end{align}
and for $c=N$, 
\begin{align}
G^{B}_{0,N}=0
\end{align}

\end{remark}

A forward/backward $CM_L$/$CM_F$ model of a reciprocal (Markov) sequence is called a reciprocal (Markov) forward/backward $CM_L$/$CM_F$ model.

\section{Determination of Explicitly Equivalent Dynamic Models: A Unified Approach}\label{General_Approach}

Let $[x_k]$ be a CM sequence governed by a model presented in Section \ref{Definitions}. We have
\begin{align}
T x&=\xi, \quad \xi  =[\xi_{0}' , \ldots , \xi_{N}']' \label{Model}
\end{align}
where the vector $\xi$ includes the dynamic noise and the boundary values and $P=\text{Cov}(\xi)$. Matrix $T$ is determined by parameters of the model. Note that $T$ and $P$ for a model have a specific form. So, parameters of equivalent models (in terms of each other) can be determined easily (see the first step of Proposition \ref{Equivalent_Construction_Lemma} below). $T$ is nonsingular for the forward and the backward $CM_L$, $CM_F$, and Markov models. Also, since the sequence is assumed to be nonsingular, $T$ is nonsingular for the reciprocal model, too \cite{Levy_Dynamic}. 

By Definition \ref{Explicit_Equivalent}, explicit equivalence is mutual, i.e., if model 2 is explicitly equivalent to model 1, then model 1 is also explicitly equivalent to model 2. We have the following proposition for the construction of explicitly equivalent models.

\begin{proposition}\label{Equivalent_Construction_Lemma}
Let a forward, backward $CM_L$, $CM_F$, reciprocal, or Markov model be given by
\begin{align}
T_1x=\xi \label{Model_1}
\end{align}
where $\xi =[\xi_{0}' ,  \ldots , \xi_{N}']'$ includes the dynamic noise and the boundary values, and the covariance of $\xi$ is $P_1$. Any explicitly equivalent forward or backward $CM_L$/$CM_F$/reciprocal/Markov model, denoted by
\begin{align}
T_2x=\zeta \label{Model_2}
\end{align}
with $\zeta=[\zeta_{0}' ,  \ldots , \zeta_{N}']'$ containing the corresponding dynamic noise and boundary values with covariance $P_2$, is constructed in two steps:

(1) The parameters of the equivalent model are determined based on\footnote{Due to the special structures of $T_1$, $P_1$, $T_2$, and $P_2$, parameters of model 2 can be easily obtained in terms of parameters of model 1 using $\eqref{Step_1}$. Then, having the parameters of model 2, $T_2$ and $P_2$ are known.}
\begin{align}
T_2'P_2^{-1}T_2=
T_1'P_1^{-1}T_1\label{Step_1}
\end{align}
Therefore, $T_2$ and $P_2$ can be obtained given their forms.

(2) The relationship between the dynamic noises and the boundary values of the two models are obtained through
\begin{align}
T_2'(P_2)^{-1}\zeta = T_1'(P_1)^{-1}\xi \label{Step_2}
\end{align}

\end{proposition}
\begin{proof}
Step (1): The inverse of the covariance matrix ($C^{-1}$) of the sequence governed by model $\eqref{Model_1}$ is calculated as $E[(T_1x)(T_1x)']=E[\xi \xi ']$, which leads to
\begin{align}\label{Inverse_C_1}
C^{-1}=T_1'(P_1)^{-1}T_1
\end{align}

The inverse of the covariance matrix of the sequence governed by $\eqref{Model_2}$ can be calculated as
\begin{align}\label{Inverse_C_2}
C^{-1}=T_2'(P_2)^{-1}T_2
\end{align}

In order for the two models to be explicitly equivalent, their governed sequences must have the same covariance matrix; thus we have $\eqref{Step_1}$ (i.e., the two models are equivalent). Due to the special structures of $T_1$, $P_1$, $T_2$, and $P_2$, parameters of model 2 can be easily obtained in terms of parameters of model 1 using $\eqref{Step_1}$ ($\eqref{MfbP_1}$--$\eqref{MfbP_5}$ below show such calculations for equivalent Markov models). Then, $P_2$ and $T_2$ are known. Note that parameters of model 2 calculated by $\eqref{Step_1}$ (in terms of those of model 1) are unique. It can be easily verified based on $\eqref{Step_1}$ for all models. This uniqueness can be also concluded from the definition of conditional expectation. Now, given $P_2$ and $T_2$, the relation between $\xi$ and $\zeta$ should be determined so that the two models are \textit{explicitly} equivalent.

Step (2): Let $P _2$ and $T _2$ be given. We show how $\eqref{Step_2}$ leads to an explicitly equivalent model. First, we show that $\zeta$ obtained by $\eqref{Step_2}$ has the desired dynamic noise and boundary values, i.e., its covariance is $P_2$. By $\eqref{Step_2}$, we have
\begin{align*}
T_2'&(P_2)^{-1}\text{Cov}(\zeta)(P_2)^{-1}T_2\\ &\quad \quad \quad =T_1'(P_1)^{-1}\text{Cov}(\xi )(P_1)^{-1}T_1
\end{align*}
Then, substituting $\text{Cov}(\xi)=P_1$, we obtain
\begin{align*}
\text{Cov}(\zeta)&=P_2(T_2')^{-1}T_1'(P_1)^{-1}P_1
(P_1)^{-1}T_1(T_2)^{-1}P_2\\
&=P_2(T_2')^{-1}T_1'
(P_1)^{-1}T_1(T_2)^{-1}P_2
\end{align*}
Since $T_1'P_1^{-1}T_1=
T_2'P_2^{-1}T_2$, we have $\text{Cov}(\zeta)=P_2(T_2')^{-1}T_2'(P_2)^{-1}T_2(T_2)^{-1}P_2=P_2$,
which means that $\zeta$ has the desired dynamic noise and boundary values. 

Second, we show that assuming $\eqref{Step_2}$ holds, two models $\eqref{Model_1}$ and $\eqref{Model_2}$ generate the same sample path of the sequence. Substituting $\eqref{Model_1}$ into $\eqref{Step_2}$, we obtain $T_1'(P_1)^{-1}T_1 x = T_2'(P_2)^{-1}\zeta$. Then, $P_2(T_2')^{-1}(T_1'(P_1)^{-1}T_1) x = \zeta \label{1}$. Pre-multiplying both sides by $(T _2)^{-1}$, we have
\begin{align}
(T_2'(P_2)^{-1}T_2)^{-1}(T_1'(P_1)^{-1}T_1) x = (T_2)^{-1}\zeta \label{2}
\end{align}
Since $T_1'P_1^{-1}T_1=
T_2'P_2^{-1}T_2$, we get $T_2 x = \zeta$. Therefore, $\eqref{Model_2}$ and $\eqref{Model_1}$ are explicitly equivalent if $\eqref{Step_2}$ holds. 
\end{proof}

The first step of Proposition \ref{Equivalent_Construction_Lemma}, $\eqref{Step_1}$, is to determine an equivalent model, and the second step, $\eqref{Step_2}$, makes the model explicitly equivalent. 

By Proposition \ref{Equivalent_Construction_Lemma}, given a model, one can construct an explicitly equivalent model. Assume that two equivalent models generate the same sample path of the governed sequence. What is the relation between sample paths of their dynamic noises and boundary values? The next proposition answers this question.

\begin{proposition}\label{Equivalent_Relation_Proposition}
Let two equivalent forward, backward $CM_L$, $CM_F$, reciprocal, or Markov models be given as
\begin{align}
T_1x=\xi \label{Dynamic_1}\\
T_2x=\zeta \label{Dynamic_2}
\end{align}
where $\xi=[\xi_{0}' , \ldots ,\xi_{N}']'$ and $\zeta=[\zeta_{0}' , \ldots ,\zeta_{N}']'$ contain their dynamic noises and boundary values, with covariances $P_1$ and $P_2$, respectively, and the nonsingular matrices $T_1$ and $T_2$ are determined by the parameters of the corresponding models. If the two models generate the same sample path of the governed sequence, the relationship between sample paths of their dynamic noises and boundary values is as $\eqref{Step_2}$.

\end{proposition}

\begin{remark}
By $\eqref{Step_1}$, $\eqref{Step_2}$ is equivalent to
\begin{align}\label{Condition_Equivalent_2}
T_1^{-1}\xi  = T_2^{-1}\zeta
\end{align}

\end{remark}

Although $\eqref{Condition_Equivalent_2}$ looks simpler, for the construction of explicitly equivalent models, $\eqref{Step_2}$ is preferred. This is explained as follows. It can be seen that the matrices $P_i$, $i=1,2$, in $\eqref{Step_2}$ corresponding to the forward or backward $CM_L$, $CM_F$, and Markov models are block-diagonal, and their inverses are easily calculated. Also, for the reciprocal model, no calculation is needed because we have $P=T$ \cite{Levy_Dynamic}. However, calculation of the inverse of $T_i$, $i=1,2$, in $\eqref{Condition_Equivalent_2}$ is not straightforward in general.

Note that Proposition \ref{Equivalent_Construction_Lemma} and Proposition \ref{Equivalent_Relation_Proposition} are not restricted to (forward/backward) $CM_L$, $CM_F$, reciprocal, and Markov models. The results work for other models satisfying the required conditions.

\section{Examples of Explicitly Sample-Equivalent Models}\label{Examples}

For illustration, two examples of explicitly sample-equivalent models obtained using the approach of Proposition \ref{Equivalent_Construction_Lemma} are presented.

\subsection{Forward and Backward Markov Models}

Given a forward Markov model $\eqref{Markov_Dynamic_Forward}$--$\eqref{M_BC}$ for $[x_k]$, by $\eqref{Step_1}$ parameters of a backward Markov model $\eqref{Markov_Dynamic_Backward}$--$\eqref{MB_BC}$ for $[x_k]$ are obtained in terms of those of the forward one as follows. For $k=2, 3, \ldots, N$,
\begin{align}
(M_0^B)^{-1}=&M_0^{-1}+M_{1,0}'M_1^{-1}M_{1,0}\label{MfbP_1}\\
M_{0,1}^B=&M_0^BM_{1,0}'M_1^{-1}\label{MfbP_2}\\
(M_{k-1}^B)^{-1}=&M_{k-1}^{-1}+M_{k,k-1}'M_{k}^{-1}M_{k,k-1}- \nonumber \\
&(M_{k-2,k-1}^B)'(M^B_{k-2})^{-1}M_{k-2,k-1}^B \label{MfbP_3}
\end{align}
\begin{align}
M^B_{k-1,k}=&M^B_{k-1}M_{k,k-1}'M_{k}^{-1} \label{MfbP_4}\\
(M^B_N)^{-1}=&M_{N}^{-1}-(M^B_{N-1,N})'(M^B_{N-1})^{-1}M^B_{N-1,N}\label{MfbP_5}
\end{align}

Then, by $\eqref{Step_2}$, the relationship between dynamic noises and boundary values of the two models is:
\begin{align}
(M^B_0)^{-1}e^{MB}_0=&M_0^{-1}e^M_0-M_{1,0}'M_1^{-1}e^M_1\label{Mfb_1}\\
(M^B_{k})^{-1}e^{MB}_{k}=&(M^B_{k-1,k})'(M^B_{k-1})^{-1}e^{MB}_{k-1} + M_{k}^{-1}e^M_{k} \nonumber\\
-&M_{k+1,k}'M_{k+1}^{-1}e^M_{k+1}, k \in [1,N-1] \label{Mfb_2}\\
(M^B_N)^{-1}e^{MB}_N=&(M^B_{N-1,N})'(M^B_{N-1})^{-1}e^{MB}_{N-1}+
M_{N}^{-1}e^M_{N}\label{Mfb_3}
\end{align} 

By the above equations, given a backward model, one can also obtain the explicitly equivalent forward model. $\eqref{MfbP_1}$--$\eqref{MfbP_5}$ and $\eqref{Mfb_1}$--$\eqref{Mfb_3}$ give explicitly equivalent forward and backward models no matter if the state transition matrix is singular or nonsingular. Based on $\eqref{Mfb_1}$--$\eqref{Mfb_3}$, we can verify the required condition in the derivation of the two-filter smoother \cite{Wax_Kailath}--\cite{Alan_Wilskey}, for singular and nonsingular state transition matrices.

\subsection{Reciprocal $CM_L$ and Reciprocal Models}

Let a $CM_L$ model governing a reciprocal sequence $[x_k]$ be given. Taking the first step of Proposition \ref{Equivalent_Construction_Lemma}, parameters of the reciprocal model governing $[x_k]$ are obtained from parameters of the $CM_L$ model as follows.
\begin{align}
R^0_0&=G_0^{-1}+G_{1,0}'G_1^{-1}G_{1,0}\label{CML_1}\\
R^0_k&=G_k^{-1}+G_{k+1,k}'G_{k+1}^{-1}G_{k+1,k}\label{CML2}, k \in [1,N-2] \\
R^0_{N-1}&=G_{N-1}^{-1}\label{CML3}\\
R^0_{N}&=G_N^{-1}+\sum _{k=1}^{N-1} G_{k,N}'G_k^{-1}G_{k,N} + G_{0,N}'G_0^{-1}G_{0,N}\label{CML_2}\\
R^+_k&=G_{k+1,k}'G_{k+1}^{-1}, k \in [0,N-2] \label{CML5}\\
R^+_{N-1}&=G_{N-1}^{-1}G_{N-1,N}\label{CML6}\\
R^-_0&=G_0^{-1}G_{0,N}-G_{1,0}'G_1^{-1}G_{1,N}\label{CML_3}
\end{align}

Then, taking the second step of Proposition \ref{Equivalent_Construction_Lemma}, the relationship between dynamic noises and boundary values of the two models is as follows. For $\eqref{CML_Dynamic_Forward}$--$\eqref{CML_Forward_BC2}$, we have
\begin{align}
e^R_0=&G_0^{-1}e_0-G_{1,0}'G_1^{-1}e_1\label{LR2_1}\\
e^R_k=&G_k^{-1}e_k-G_{k+1,k}'G_{k+1}^{-1}e_{k+1}, k \in [1,N-2]\label{LR_2}\\
e^R_{N-1}=&G_{N-1}^{-1}e_{N-1}\label{LR_3}\\
e^R_N=&-\sum _{k=1}^{N-1}G_{k,N}'G_k^{-1}e_k+G_N^{-1}e_N-G_{0,N}'G_0^{-1}e_0\label{LR2_4}
\end{align}

By the above equations, given a reciprocal model, one can obtain the explicitly equivalent reciprocal $CM_L$ model. This is important because the reciprocal $CM_L$ model can be more easily applicable than the reciprocal model.

\section{Summary and Conclusions}\label{Summary}

Relationships between dynamic models governing different classes of Gaussian conditionally Markov (CM) sequences (including Markov, reciprocal, and so-called $CM_L$ and $CM_F$ sequences) have been studied. One CM sequence can obey different models. Given one, it is desirable to obtain other models for the same sequence. Two models are called \textit{equivalent} if they govern the same random sequence (i.e., their sequences have the same distribution). In some problems it is not sufficient to have only equivalent models---we need to know the relationship between dynamic noises and boundary values for the equivalent models to have the same sample path of the governed sequence. Two equivalent models are \textit{explicitly sample-equivalent} if such a relationship is given. 

A unified approach has been presented, such that given a forward, backward $CM_L$, $CM_F$, reciprocal, or Markov model, any explicitly equivalent such model can be obtained. This approach does not require any assumptions (e.g., nonsingularity) about the matrix coefficients of the models. So, unlike \cite{Verghese} (which is restricted to nonsingular state transition matrices), the proposed approach can be used to obtain a backward Markov model explicitly equivalent to a forward Markov model with either a singular or nonsingular state transition matrix.

\subsubsection*{Acknowledgments}

Research was supported by NASA Phase03-06 through grant NNX13AD29A.


\begin{thebibliography}{99}

\medskip

\small


\bibitem{Levy_1} B. Levy and A. J. Krener. Dynamics and Kinematics of Reciprocal Diffusions,” \textit{Journal of Math. Physics}. Vol. 34, No. 5, pp. 1846-1875, 1993. 
\bibitem{Levy_2} B. Levy and A. J. Krener. Stochastic Mechanics of Reciprocal Diffusions. \textit{Journal of Math. Physics}. Vol. 37, No. 2, pp. 769-802, 1996. 
\bibitem{Fanas1} M. Fanaswala, V. Krishnamurthy, and L. B. White. Destination-aware Target Tracking via Syntactic Signal Processing. \textit{IEEE International Conference on Acoustics, Speech and Signal Processing}, May 2011.
\bibitem{Fanas2} M. Fanaswala and V. Krishnamurthy. Detection of Anomalous Trajectory Patterns in Target Tracking via Stochastic Context-Free Grammer and Reciprocal Process Models. \textit{IEEE Journal of Selected Topics in Signal Processing}, Vol. 7, No. 1, pp. 76-90, 2013.
\bibitem{Simon} B. I. Ahmad, J. K. Murphy, S. J. Godsill, P. M. Langdon, and R. Hardy. Intelligent Interactive Displays in Vehicles with Intent Prediction: A Bayesian Framework. \textit{IEEE Signal Processing Magazine}, Vol. 34, No. 2, pp. 82-94, 2017.
\bibitem{Picci} A. Chiuso, A. Ferrante, and G. Picci. Reciprocal Realization and Modeling of Textured Images. \textit{44th IEEE Conference on Decision and Control}, Dec. 2005. 
\bibitem{Picci2} G. Picci and F. Carli. Modelling and Simulation of Images by Reciprocal Processes. \textit{Tenth International Conference on Computer Modeling and Simulation}, Apr. 2008.
\bibitem{DD_Conf} R. Rezaie and X. R. Li. Destination-Directed Trajectory Modeling and Prediction Using Conditionally Markov Sequences. \textit{IEEE Western New York Image and Signal Processing Workshop}, Rochester, NY, USA, Oct. 2018, pp. 1-5.
\bibitem{DW_Conf} R. Rezaie and X. R. Li. Trajectory Modeling and Prediction with Waypoint Information Using a Conditionally Markov Sequence. \textit{56th Allerton Conference on Communication, Control, and Computing}, Monticello, IL, USA, Oct. 2018, pp. 486-493.
\bibitem{Krener1} A. J. Krener. Reciprocal Processes and the Stochastic Realization Problem for Acausal Systems. \textit{Modeling, Identification, and Robust Control}, C. I. Byrnes and A. Lindquist (editors), Elsevier, 1986.
\bibitem{Mehr} C. B. Mehr, J.  A. McFadden, Certain Properties of Gaussian Processes and Their First-Passage Times. J. Roy. Stat. Soc., (B) vol. 27, pp. 505-522, 1965.
\bibitem{Bernstein} S. Bernstein. Sur les liaisons entre les grandeurs aleatoires. \textit{Verh. des intern. Mathematikerkongr I,} Zurich, 1932.
\bibitem{Schrodinger_1} E. Schrodinger. Uber die Umkehrung der Naturgesetze. \textit{Sitz. Ber. der Preuss. Akad. Wissen., Berlin Phys. Math.} 144, 1931.
\bibitem{Schrodinger_2} E. Schrodinger. Theorie relativiste de l'electron et l'interpretation de la mechanique quantique. \textit{Ann. Inst. H. Poincare} 2, pp. 269-310, 1932.
\bibitem{Slepian} D. Slepian. First Passage Time for a Particular Gaussian Process. \textit{Ann. Math. Statist.} 32, pp. 610-612, 1961.
\bibitem{Jamison_1} B. Jamison. Reciprocal Processes: The Stationary Gaussian Case. \textit{Ann. Math. Statist.}, vol.41, No. 5, pp. 1624-1630, 1970.
\bibitem{Chay} S. C. Chay, On Quasi-Markov Random Fields, \textit{Journal of Multivariate Analysis} 2, pp. 14-76, 1972.
\bibitem{Jamison_Reciprocal} B. Jamison. Reciprocal Processes. \textit{Z. Wahrscheinlichkeitstheorie verw. Gebiete}, vol. 30, pp. 65-86, 1974. 
\bibitem{Reciprocal_Measure} C. Leonard, S. Rœlly, and J-C Zambrini. Reciprocal Processes. A Measure-theoretical Point of View. \textit{Probability Surveys}, Vol. 11, pp. 237–269, 2014.
\bibitem{Conforti} G. Conforti, P. Dai Pra, and S. Roelly. Reciprocal Class of Jump Processes. \textit{Journal of Theoretical Probability}, June 2014.
\bibitem{Murr} R. Murr, \textit{Reciprocal Classes of Markov Processes. An Approach with Duality Formulae}, Ph.D. thesis, Universitat Potsdam, 2012.
\bibitem{Roally} S. Rœlly. \textit{Reciprocal Processes. A Stochastic Analysis Approach}. In V. Korolyuk, N. Limnios, Y. Mishura, L. Sakhno, and G. Shevchenko, editors, Modern Stochastics and Applications, volume 90 of Optimization and Its Applications, pp. 53–67. Springer, 2014.
\bibitem{Corrected_Stationary} J-P Carmichael, J-C Masse, and R. Theodorescu, Processus Gaussiens Stationnaires Reciproques sur un Intervalle, \textit{C. R. Acad. Sc. Paris}, t. 295 (27 Sep. 1982).
\bibitem{ABRAHAM} J. Abraham and J. Thomas. Some Comments on Conditionally Markov and Reciprocal Gaussian Processes. \textit{IEEE Trans. on Information Theory}. Vol. 27, No. 4, pp. 523-525, 1981.
\bibitem{Krener_1} A. J. Krener. Reciprocal Diffusions and Stochastic Differential Equations of Second Order. \textit{Stochastics}, Vol. 24, No. 4, pp. 393-422, 1988.
\bibitem{Levy_Dynamic} B. C. Levy, R. Frezza, and A. Krener. Modeling and Estimation of Discrete-Time Gaussian Reciprocal Processes. \textit{IEEE Trans. on Automatic Control}. Vol. 35, No. 9, pp. 1013-1023, 1990.
\bibitem{Krener_2} A. J. Krener, R. Frezza, and B. C. Levy. Gaussian Reciprocal Processes and Self-adjoint Stochastic Differential Equations of Second Order. \textit{Stochastics and Stochastic Reports}, Vol. 34, Nos. 1-2, pp. 29-56, 1991. 
\bibitem{Chen} J. Chen and H. L. Weinert, A New Characterization of Multivariate Gaussian Reciprocal Processes. \textit{IEEE Trans. on Automatic Control}, Vol. 38, No. 10, pp. 1601-1602, 1993. 
\bibitem{Carli} F. P. Carli, A. Ferrante, M. Pavon, G. Picci. A Maximum Entropy Solution of the Covariance Extension Problem for Reciprocal Processes. \textit{IEEE Trans. on Automatic Control}, Vol. 56, No. 9, pp. 1999-2012, 2011.
\bibitem{Carra} F. Carravetta. Nearest-neighbour Modelling of Reciprocal Chains. \textit{An International Journal of Probability and Stochastic Processes}, Vol. 80, No. 6, pp. 525-584, 2008.
\bibitem{White_2} L. B. White and F. Carravetta. Optimal Smoothing for Finite State Hidden Reciprocal Processes. \textit{IEEE Trans. on Automatic Control}, Vol. 56, No. 9, 2156-2161, 2011. 
\bibitem{White} F. Carravetta and L. B. White. Modelling and Estimation for Finite State Reciprocal Processes. \textit{IEEE Trans. on Automatic Control}, Vol. 57, No. 9, pp. 2190-2202, 2012.
\bibitem {White_3} L B. White and H. X. Vu. Maximum Likelihood Sequence Estimation for Hidden Reciprocal Processes. \textit{IEEE Trans. on Automatic Control}, Vol. 58, No. 10, pp. 2670-2674, 2013.
\bibitem{CM_Part_II_A_Conf} R. Rezaie and X. R. Li. Gaussian Reciprocal Sequences from the Viewpoint of Conditionally Markov Sequences. \textit{Inter. Conference on Vision, Image and Signal Processing}, Las Vegas, NV, USA, Aug. 2018, pp. 33:1-33:6.
\bibitem{CM_Part_II_B_Conf} R. Rezaie and X. R. Li. Models and Representations of Gaussian Reciprocal and Conditionally Markov Sequences. \textit{Inter. Conference on Vision, Image and Signal Processing}, Las Vegas, NV, USA, Aug. 2018, pp. 65:1-65:6.

\bibitem{CM_Part_III_Journal} Rezaie, R. and Li, X. R., Gaussian Conditionally Markov Sequences: Dynamic Models and Representations of Reciprocal and Other Classes. \textit{IEEE T-SP}.
\bibitem{CM_SingularNonsingular} R. Rezaie and X. R. Li. Gaussian Conditionally Markov Sequences: Singular/Nonsingular. \textit{IEEE T-AC}.
\bibitem{Thesis_Reza} R. Rezaie. \textit{Gaussian Conditionally Markov Sequences: Theory with Application}. Ph.D. Dissertation, Dept of Electrical Engineering, University of New Orleans.
\bibitem{CM_Journal_Algebraically} R. Rezaie and X. R. Li. Gaussian Conditionally Markov Sequences: Algebraically Equivalent Dynamic Models. \textit{IEEE T-AES}.

\bibitem{Moura} D. Vats and J. M. F. Moura. Recursive Filtering and Smoothing for Discrete Index Gaussian Reciprocal Processes. \textit{43rd Annual Conference on Information Sciences and Systems}, Mar. 2009.
\bibitem{CM_Part_I_Conf} R. Rezaie and X. R. Li. Nonsingular Gaussian Conditionally Markov Sequences. \textit{IEEE West. New York Image and SP Workshop}. Rochester, USA, Oct. 2018, pp. 1-5.
\bibitem{Wax_Kailath} Mati Wax and Thomas Kailath. Direct Approach to Two-filter Smoothing Formulas. \textit{International Journal of Control}. Vol. 39, No. 3, 1984.
\bibitem{Fraser} D. Fraser and J. Potter. The Optimum Linear Smoother as a Combination of Two Optimum Linear Filters. \textit{IEEE Trans. on Automatic Control}. Vol. 14, No. 4, pp. 387-390, 1969. 
\bibitem{Alan_Wilskey} J. E. Wall, A. Willsky, and N. R. Sandell. On the Fixed-Interval Smoothing Problem. \textit{Stochastics}. Vol. 5, pp. 1-41, 1981. 
\bibitem{MB_1} B. Friedlander, T. Kailath, and L. Ljung, Scattering Theory and Linear Least Squares Estimation, Part II: Discrete-time Problems. \textit{J. Franklin Inst.}, Vol. 301, pp. 71-82, 1976.
\bibitem{MB_2} L. Ljung and T. Kailath, Backwards Markovian Models for Second-order Stochastic Processes. \textit{IEEE Trans. on Information Theory}, Vol. IT-22, No. 4, pp. 488-491, 1979.
\bibitem{MB_3} G. S. Sidhu and U. B. Desai, New Smoothing Algorithms Based on Reversed-time Lumped Models. \textit{IEEE Trans. on Automatic Control}, Vol. AC-21, No. 4, pp. 538-541, 1976.
\bibitem{MB_4} D. G. Lainiotis, General Backwards Markov Models. \textit{IEEE Trans. Automatic Control} Vol. AC-21, No. 4, pp. 595-599, 1976.  
\bibitem{Verghese} G. Verghese and T. Kailath, A Further Note on Backward Markovian Models. \textit{IEEE Trans. on Information Theory}, Vol. IT-25, No. 1, pp. 121-124, 1979.


\end{thebibliography}
\end{document}